\documentclass[11pt]{amsart}
\usepackage{amssymb, latexsym, amsmath, amsfonts, tikz, tikz-cd}
\usepackage{hyperref, enumitem, fancyhdr}

\newtheorem{thm}{Theorem}[section]
\newtheorem{cor}[thm]{Corollary}
\newtheorem{lem}[thm]{Lemma}
\newtheorem{prop}[thm]{Proposition}
\theoremstyle{definition}
\newtheorem{defn}[thm]{Definition}
\theoremstyle{remark}

\numberwithin{equation}{section}
\theoremstyle{remark}
\newtheorem{exam}[thm]{Example}

\newcommand{\mbb}{\mathbb}
\newcommand{\ra}{\rightarrow}

\newcommand{\pa}{\partial}
\newcommand{\ov}{\overline}
\newcommand{\sm}{\setminus}
\newcommand{\ep}{\epsilon}
\newcommand{\no}{\noindent}
\newcommand{\al}{\alpha}

\newcommand{\cal}{\mathcal}
\newcommand{\ti}{\tilde}

\begin{document}
\title{Regularity of the leafwise Poincar\'{e} metric on singular holomorphic foliations}

\thanks{*The author is supported by the Labex CEMPI (ANR-11-LABX-0007-01)}

\author{Sahil Gehlawat*, Kaushal Verma}

\address{SG: Universit\'{e} de Lille, Laboratoire de Math\'{e}matiques Paul Painlev\'{e}, CNRS U.M.R. 8524, 59655
Villeneuve d’Ascq Cedex, France.}
\email{sahil.gehlawat@univ-lille.fr}

\address{KV: Department of Mathematics, Indian Institute of Science, Bangalore 560 012, India}
\email{kverma@iisc.ac.in}

\dedicatory{In Memoriam: Nessim Sibony}

\begin{abstract} 
Let $\mathcal F$ be a smooth Riemann surface foliation on $M \setminus E$, where $M$ is a complex manifold and the singular set $E \subset M$ is an analytic set of codimension at least two. Fix a hermitian metric on $M$ and assume that all leaves of $\mathcal F$ are hyperbolic. Verjovsky's modulus of uniformization $\eta$ is a positive real function defined on $M \setminus E$ defined in terms of the family of holomorphic maps from the unit disc $\mathbb D$ into the leaves of $\mathcal F$ and is a measure of the largest possible derivative in the class of such maps. Various conditions are known that guarantee the continuity of $\eta$ on $M \setminus E$. The main question that is addressed here is its continuity at points of $E$. To do this, we adapt Whitney's $C_4$-tangent cone construction for analytic sets to the setting of foliations and use it to define the tangent cone of $\mathcal F$ at points of $E$. This leads to the definition of a foliation that is of {\it transversal type} at points of $E$. It is shown that the map $\eta$ associated to such foliations is continuous at $E$ provided that it is continuous on $M \setminus E$ and $\mathcal F$ is of transversal type. We also present observations on the locus of discontinuity of $\eta$. Finally, for a domain $U \subset M$, we consider $\mathcal F_U$, the restriction of $\mathcal F$ to $U$ and the corresponding positive function $\eta_U$. Using the transversality hypothesis leads to strengthened versions of the results of Lins Neto--Martins on the variation $U \mapsto \eta_U$. 
\end{abstract}  

\maketitle 

\section{Introduction}

\noindent Let $\mathcal{F}$ be a singular holomorphic foliation by curves ({\sf shfc}) on a complex manifold $M$ with singular set $E \subset M$. The singular set $E \subset M$ is a complex analytic set of codimension at least $2$. For $p \in M \setminus E$, let $L_p$ be the leaf passing through $p$. The tangent line to $\mathcal F$ at $p$, denoted by $T_p\mathcal F$, consists of all vectors tangent to $L_p$ at $p$. The union of all such tangent lines 
\[
T \mathcal{F} := \bigcup_{p \in M \setminus E} T_{p} \mathcal{F}
\]
defines a holomorphic line bundle on $M\setminus E$. To fix notation, let ${\sf reg}(A)$ and ${\sf sng}(A)$ denote the smooth and singular loci, respectively, of a complex analytic set $A$. The tangent space at points of ${\sf reg}(A)$ is well defined; on the other hand, according to Whitney (see \cite{Ch}, \cite{Wh}), there are six different possible definitions of the tangent cone at points in ${\sf sng}(A)$. Each of these descriptions coincides with the (unambiguous) tangent space when applied to points in ${\sf reg}(A)$. We will focus on what is called Whitney's $C_4$-tangent cone in \cite{Ch} (see \cite{Wh} also) by first recalling its definition as applied to $E$. For $p \in {\sf sng}(E)$, let
\[
C_{p}E = \left\{v \in T_p M \mid \text{there exists} \ \{q_n\}_{n \ge 1} \subset {\sf reg}(E), v_n \in T_{q_n} E \ \text{such that} \ (q_n, v_n) \rightarrow (p, v) \right\}.
\]
At points $p \in {\sf reg}(E)$, the tangent cone to $E$ is just the tangent space $T_pE \subset T_pM$. This leads to the consideration of: 

\begin{defn}\label{D:TangentCone}
Let $\mathcal{F}$ be a {\sf shfc} on a complex manifold $M$ with singular set $E \subset M$. Let $p \in E$. Define the \textit{tangent cone of $\mathcal{F}$ at $p$} as
\[
C_{p} \mathcal{F} = \left\{v \in T_p M \mid \text{there exists} \ \{p_n\}_{n \ge 1} \subset M \setminus E, v_n \in T_{p_n} \mathcal{F} \ \text{such that} \ (p_n, v_n) \rightarrow (p, v)\right\}. 
\]
\end{defn}
For a given $v \in C_p \mathcal F$ and $\lambda \in \mathbb C$, let $(p_n, v_n) \in (M \setminus E) \times T_{p_n} \mathcal F$ converge to $(p, v)$. Note that $\lambda v_n \in T_{p_n} \mathcal F$ for all $n$ and that $(p_n, \lambda v_n) \in (M \setminus E) \times T_{p_n} \mathcal F$ converges to $(p, \lambda v)$. This shows that $C_p \mathcal F$ is actually a cone with vertex at $p$.

\medskip

A basic question is to understand how the leaves of a given foliation accumulate on its singular set. Keeping this in mind, we consider:

\begin{defn}\label{D:Transversality}
Let $\mathcal{F}$ be a {\sf shfc} on a complex manifold $M$ with singular set $E \subset M$. Say that $\mathcal{F}$ is of \textit{transversal type at $p \in E$} if there exists a neighourhood $U_{p} \subset M$ of $p$ such that for all $q \in U_{p} \cap E$
\[
\overline{C_{q} \mathcal{F}} \cap C_{q} E = \{0\}.
\]

\no We will say that $\mathcal{F}$ is of \textit{transversal type} if it is of transversal type at each $p \in E$. The following examples will be helpful in illustrating this notion.
\end{defn}

\begin{exam}\label{E:E1.3}
Let $X$ be a holomorphic vector field defined on the bidisc $M = \mbb{D} \times \mbb{D} \subset \mbb{C}^2$, with singular set $E = X(0) = \{0\}$. Let $\lambda_1, \lambda_2$ be the eigenvalues of the linear part $X^{(1)}$ of the vector field $X$, and suppose that 
$\alpha = {\lambda_1}/{\lambda_2} \neq 0$ is its index. Consider the following local normal forms of $X$ near the origin:
\begin{enumerate}

\item If $\alpha \notin \mbb{R}^{-} \cup \{0\}$ and $\alpha, \alpha^{-1} \notin \mbb{N}$,
\[
X = x\frac{\partial}{\partial x} + \alpha y \frac{\partial}{\partial y}.
\]

\item If $\alpha \in \mbb{R}^{-} \sm \{0\}$,
\[
X = x\frac{\partial}{\partial x} + \alpha y(1+ f(x,y)) \frac{\partial}{\partial y},
\]
where $f(x,y)$ is a holomorphic function defined on a neighourhood of the origin with the property that both $x,y \mid f(x,y)$.

\item If $\alpha = n \in \mbb{N}$ or $\alpha^{-1} = n \in \mbb{N}$,
\[
X = x\frac{\partial}{\partial x} + (ny + ax^n) \frac{\partial}{\partial y} \; \;  \text{or} \; \; X = (nx+ ay^n)\frac{\partial}{\partial x} + y \frac{\partial}{\partial y}
\]
respectively, where $a \in \mathbb C$ is a constant.


\end{enumerate}

\medskip

\no Let us now show that $C_{0}\cal{F}_{X} = \mbb{C}^2$ in each of these cases. 

\medskip
\begin{enumerate}

\item Let $V = (v_1,v_2) \in \mbb{C}^2 \sm \{0\}$. Define $z_m = (\frac{v_1}{m}, \frac{v_2}{\alpha m})$ for $m \ge 1$. Then $z_m$ converges to the origin and $V_m = m X(z_m) \in T_{z_m}\cal{F}_{X}$. But note that
\[
m X(z_m) = m \Big(\frac{v_1}{m}\frac{\partial}{\partial x} + \alpha \frac{v_2}{\alpha m} \frac{\partial}{\partial y}\Big) = V
\]
and this shows that $V \in C_{0}\cal{F}_{X}$. 

\medskip

\item Let $V=(v_1, v_2) \in \mbb{C}^2 \sm \{0\}$. For each $m \in \mbb{N}$,  
\[
h_{m}(y) = y \Big(1 + f\Big(\frac{v_1}{m}, y\Big)\Big)
\]
is holomorphic in $y$. Since $f(0,0)= 0$, let $P(r) \subset U$ be the $r-$radius bidisc such that $\vert f(x,y)\vert < {1}/{2}$, for all $(x,y) \in P(r)$. For large enough $m$ and $y \in D(0,r)$, note that $(\frac{v_1}{m}, y) \in P(r)$. Therefore,
\[
\vert h_{m}(y)\vert^2 > \vert y\vert^2 \; \Big(1 - \vert f\Big(\frac{v_1}{m}, y\Big)\vert^2\Big) > \frac{1}{2} \vert y\vert^2.
\]

This tells us that there exist $0 < r_1 < r$ such that $D(0,r_1) \subset h_{m}(D(0,r))$ for all $m$ large enough. Choose $N \ge 1$ such that $D(0,r_1) \subset h_{m}(D(0,r))$ and $\frac{v_2}{\alpha m} \in D(0, r_1)$ for all $m \ge N$. Let $y_m \in D(0, r)$ be such that $h_m(y_m) = \frac{v_2}{\alpha m}$. Observe that
\[
0 \le \frac{\vert y_m\vert^2}{2} < \vert h_m(y_m)\vert^2 = \Big(\frac{v_2}{\alpha m}\Big)^2 \to 0,
\]
and $y_m \neq 0$ if $v_2 \neq 0$. Therefore $y_m \to 0$. Take $z_{m} = (\frac{v_1}{m}, y_m) \in M \sm \{0\}$ and 
\[
V_m = m X(z_m) = m\Big(\frac{v_1}{m}\Big) \frac{\partial}{\partial x} + m  \alpha \; h_{m}(y_m) \frac{\partial}{\partial y} = (v_1, v_2) = V.
\]
Therefore, $C_{0}\cal{F}_{X} = \mbb{C}^2$.

\medskip

\item Suppose $\alpha = n \in \mbb{N}$ and $V = (v_1, v_2) \in \mbb{C}^2$. For $m \ge 1$, let 
\[
z_m = \Big(\frac{v_1}{m}, \frac{1}{m n} \Big(v_2 - \frac{a v_{1}^n}{m^{n-1}}\Big)\Big)
\]
and note that $z_m \to 0$ as $m \to +\infty$. Furthermore, note that
\[
m X(z_m) = m\Big(\frac{v_1}{m}, n \frac{1}{m n} \Big(v_2 - \frac{a v_{1}^n}{m^{n-1}}\Big) + a\Big(\frac{v_1}{m}\Big)^n\Big) = \Big(v_1, v_2 - \frac{a v_{1}^n}{m^{n-1}} + am\Big(\frac{v_1}{m}\Big)^n\Big)
\]
which shows that $m X(z_m) = (v_1, v_2) = V$ if $n\ge1$ and converges to $V$ as $m \to +\infty$. Since $V_m = m X(z_m) \in T_{z_m}\cal{F}_{X}$, it follows that $V \in C_{0}\cal{F}_{X}$. Therefore, $C_{0}\cal{F}_{X} = \mbb{C}^2$. The case $\alpha^{-1} \in \mbb{N}$ is similar.

\end{enumerate}

\end{exam}

\medskip

\begin{exam}\label{E:E1.4}
Consider the holomorphic vector field
\[
X(x,y,z) = x \frac{\partial}{\partial x} + h(z)y \frac{\partial}{\partial y}
\]
defined on the unit polydisc $M = \mathbb D^3 \subset \mbb C^3$; here, $h$ is a non-vanishing holomorphic function. In this case, $E = \{X = 0\} = \{(0,0,z) \in M\}$. 

\medskip
\no We will show that for each $p = (0,0,c) \in E$, $C_{p}\cal{F}_{X} = \langle e_1, e_2\rangle$, where $\langle e_1, e_2\rangle \backsimeq \mbb C^2$ is the complex subspace of $\mbb{C}^3$ spanned by the vectors $e_1 = (1,0,0)$ and $e_2 = (0,1,0)$. To see that $C_{p}\cal{F}_{X} \subset \langle e_1, e_2\rangle$, note that if $v = (v_1, v_2, v_3) \in T_{z}\cal{F}_{X}$ for any $z \in M \sm E$, then $v_3 = 0$. Therefore, if $V = (v_1, v_2, v_3) \in C_{p}\cal{F}_{X}$, then $v_3 = 0$ and so $V \in \langle e_1, e_2\rangle$.

\medskip
\no Now observe that $\Sigma_{c} = \{z = c\}$ is $\cal{F}_{X}-$invariant for every $c \in \mbb{D}$. Consider the vector field $X_{c} = X\vert_{\Sigma_{c}}$ 
\[
X_{c} = x \frac{\partial}{\partial x} + h(c)y\frac{\partial}{\partial y}
\]
defined on $M_{c} = \mbb{D}^2 \subset \mbb{C}^2$ with $E_{c} = \{X_{c} = 0\} = \{0\}$. Since $h(c) \neq 0$, the calculations in the previous examples show that $C_{0}\cal{F}_{X_{c}} = \mbb{C}^2 = \langle e_1, e_2\rangle$. Also, it is clear that $C_{0}\cal{F}_{X_{c}} \subset C_{p}\cal{F}_{X}$. Therefore $\langle e_1, e_2\rangle \subset C_{p}\cal{F}_{X}$. 

\medskip 
\no Finally, 
\[
C_{p}E \cap \overline{C_{p}\cal{F}_{X}} = \{0\}.
\]
since $C_{p}E = \langle e_3\rangle$ for all $p \in E$, and this shows that $\cal{F}_{X}$ is transversal type at each $p \in E$.

\end{exam}

\medskip

\begin{exam}\label{E:E1.5}
Let $\cal{F}_X$ be the holomorphic foliation induced by the vector field 
\[
X = x\frac{\partial}{\partial x} + zy \frac{\partial}{\partial y} + zy \frac{\partial}{\partial z}
\]
on the open polydisc $M \subset \mbb C^3$ around the origin with multi-radius $r = (r_1,r_2,r_3)$. The singular set of $\cal{F}_{X}$ is $E = \{(0,y,z) \in M : yz = 0\} = \{y-\text{axis}\}\cup \{z-\text{axis}\}$.

\medskip

\no We will show that for each $p \in E$, $C_{p}\cal{F}_{X} = \langle (1,0,0), (0,1,1)\rangle$, and $\cal{F}_{X}$ is transversal type.

\medskip
For $p \in E$, note that $C_{p}\cal{F}_{X} \subset \langle (1,0,0), (0,1,1)\rangle$. Indeed, for any $v = (v_1, v_2, v_3) \in T_{z}\cal{F}_{X}$ where $z \in M \sm E$, we must have $v_2 = v_3$. Therefore, $v \in \langle (1,0,0), (0,1,1)\rangle$, and so if $V \in C_{p}\cal{F}_{X}$, then $V \in \langle (1,0,0), (0,1,1)\rangle$.

\medskip
Suppose now that $V = (v_1, v_2,v_2) \in \langle (1,0,0), (0,1,1)\rangle \sm \{0\}$. There are three cases to consider:

\begin{itemize}
\item Let $p = (0,c,0) \in E$, where $c \in D(0,r_2) \setminus \{0\}$. Consider the sequence of points $z_n = (\frac{v_1}{n}, c, \frac{v_2}{c n}) \in M \sm E$, and the vectors
\[
V_n = n X(z_n) = n\Big(\frac{v_1}{n}, c \frac{v_2}{c n}, c \frac{v_2}{c n}\Big) = (v_1, v_2, v_2) = V.
\]
Since $V = V_n \in T_{z_n}\cal{F}_{X}$ and $z_n \to p$, it follows that $V \in C_{p}\cal{F}_{X}$. Hence,
\[
\langle (1,0,0), (0,1,1)\rangle = C_{p}\cal{F}_{X} = \overline{C_{p}\cal{F}_{X}}.
\]
Further, since $C_{p}E = \langle (0,1,0)\rangle$, it follows that
\[
C_{p}E \cap \overline{C_{p}\cal{F}_{X}} = \langle (0,1,0)\rangle \cap \langle (1,0,0), (0,1,1)\rangle = \{0\}.
\]
As this is true for all $p = (0, c, 0)$ with $c \in D(0,r_2) \setminus \{0\}$, take a neighbourhood $U_{p}$ of $p$ so that every $q \in U_{p} \cap E$ is of the form $q = (0,y,0)$ with $y \not= 0$. Therefore, $\cal{F}_{X}$ is transversal type at $p \in E$.

\medskip

\item Let $p = (0,0,c) \in E$, where $c \in D(0,r_3) \setminus \{0\}$. Consider the sequence of points $z_n = (\frac{v_1}{n}, \frac{v_2}{c n}, c) \in M \sm E$, and the vectors
\[
V_n = n X(z_n) = n\Big(\frac{v_1}{n}, c \frac{v_2}{c n}, c \frac{v_2}{c n}\Big) = (v_1, v_2, v_2) = V.
\]
Since $V = V_n \in T_{z_n}\cal{F}_{X}$ and $z_n \to p$, it follows that $V \in C_{p}\cal{F}_{X}$. Hence,
\[
\langle (1,0,0), (0,1,1)\rangle = C_{p}\cal{F}_{X} = \overline{C_{p}\cal{F}_{X}}.
\]
Further, since $C_{p}E = \langle (0,0,1)\rangle$, it follows that
\[
C_{p}E \cap \overline{C_{p}\cal{F}_{X}} = \langle (0,0,1)\rangle \cap \langle (1,0,0), (0,1,1)\rangle = \{0\}.
\]
As this is true for all $p = (0, 0, c)$ with $c \in D(0,r_3) \setminus \{0\}$, take a neighbourhood $U_{p}$ of $p$ so that every $q \in U_{p} \cap E$ is of the form $q = (0,0,z)$ with $z \not= 0$. Therefore, $\cal{F}_{X}$ is transversal type at $p \in E$.

\medskip
\item Let $p = (0,0,0) \in E$. In this case, consider the sequence 
\[
z_n = \left(\frac{v_1}{n}, (\frac{v_2}{n})^{1/2}, (\frac{v_2}{n})^{1/2} \right) \in M \sm E 
\]
and the vectors
\[
V_n = n X(z_n) = n\Big(\frac{v_1}{n},  \frac{v_2}{n}, \frac{v_2}{n}\Big) = (v_1, v_2, v_2) = V.
\]
Since $V = V_n \in T_{z_n}\cal{F}_{X}$ and $z_n \to 0$, it follows that $V \in C_{0}\cal{F}_{X}$. Hence,
\[
\langle (1,0,0), (0,1,1)\rangle = C_{0}\cal{F}_{X} = \overline{C_{0}\cal{F}_{X}}.
\]
\no Further, since $C_{0}E = \langle (0,0,1)\rangle \cup \langle (0,1,0)\rangle$, it follows that
\[
C_{0}E \cap \overline{C_{0}\cal{F}_{X}} = \big(\langle (0,0,1)\rangle \cup \langle (0,1,0)\rangle\big) \cap \langle (1,0,0), (0,1,1)\rangle = \{0\}.
\]
Let $U_{0} = M$. By the previous observations, the transversality conditions holds for all $p \in E$, and this implies that $\cal{F}_{X}$ is transversal type at $0 \in E$.

\end{itemize}

\no Hence $\cal{F}_{X}$ is of transversal type.

\end{exam}

\medskip

In what follows, we will assume that $\mathcal F$ is hyperbolic. Thus, each leaf $L_p \subset \mathcal F$ is a hyperbolic Riemann surface. Fix a Riemannian metric $g$ on $M$ and let $\vert v \vert_g$ denote the length of a tangent vector $v$. Let $\mathcal O(\mathbb D, \mathcal F)$ denote the family of holomorphic maps $f$ from $\mathbb D$ that take values in a leaf $L \subset \mathcal F$. The modulus of uniformization $\eta : M \setminus E \to (0, \infty)$ given by
\[
\eta(z) = \sup \{ \vert f'(0) \vert_g : f \in \mathcal O(\mathbb D, \mathcal F), f(0) = z \} 
\]
was defined by Verjovsky \cite{V} and its continuity properties under suitable hypotheses on $E$ were considered by Lins Neto \cite{Ne1, Ne2}, Lins Neto--Martins \cite{NM}, Candel \cite{Ca}, and more recently by Dinh--Nguyen--Sibony \cite{DNS1, DNS2}. This note builds on \cite{GV}, which in turn was inspired by \cite{NM}. Recall that \cite{NM} provides several equivalent sufficient conditions for the continuity of $\eta$ on $M \setminus E$ for the case when $E$ is discrete.

\medskip

The aim of this note is two fold. First, to generalize some of the results in \cite{NM} for hyperbolic {\sf shfc}'s without restrictions on $E$ and to prove that if  $\mathcal F$ is of transversal type and $\eta$ is continuous on $M \setminus E$, then $\eta$ extends continuously to all of $M$. The techniques of the proofs are similar to the ones used in \cite{NM}. Second, let
\[
D_{\mathcal F} = \{ p \in M \setminus E: \eta \; \text{is discontinuous at} \; p \}
\]
be the discontinuity locus of $\eta$. Several properties of the relative location of $D_{\mathcal F}$ in $M \setminus E$ are obtained and some of these observations strengthen the results obtained earlier by Fornaess--Sibony \cite{FS}. 

\medskip

\no As in \cite{NM}, the space of uniformizations
\[
\mathcal{U} = \{\alpha \in \mathcal{O}(\mbb{D}, \mathcal{F}) \mid \alpha \ \text{is a uniformization of a leaf of} \ \mathcal{F}\}
\]
is the essential object to consider and we will say that $\mathcal U$ is NCP (Normal on compact parts) if for any family $\mathcal{H} \subset \mathcal{U}$ such that $\{\alpha(0) \mid \alpha \in \mathcal{H}\}$ is relatively compact in $M$, the family $\mathcal{H}$ is normal. This happens if $M$ is taut, for example.


\medskip
Before stating the main results, the following lemma describes the closure of $\cal{U}$ and will be used repeatedly.

\begin{lem}\label{L:LimitUnif}
Let $\mathcal{F}$ be a hyperbolic {\sf shfc} on a complex manifold $M$ with singular set $E \subset M$ of dimension $k \ge 0$. If $\{\alpha_n\}_{n \ge 1}$ is a sequence in $\mathcal{U}$ which converges on compact subsets of $\mbb{D}$ to $\alpha : \mbb{D} \to M$, then $\alpha(\mbb{D})\subset L \cup E$, where either $L = \emptyset$ or $L$ is a leaf of $\mathcal{F}$. In particular, if $L = \emptyset$ and $q \in \alpha(\mbb{D}) \subset E$ where $\mathcal{F}$ is transversal type at $q \in E$, then $\alpha \equiv q \in E$.
\end{lem}

\begin{proof}
Consider the open set $U = \{z \in \mbb{D} : \alpha(z) \notin E\}$. Define for $z_0 \in U$, the set $U_{z_0} = \{z \in U : L_{z} = L_{z_0}\}$. Using the fact that $\alpha$ is a local immersion at the points of $U$ and local trivializations of $\mathcal F$ outside $E$, it can be seen that $U_{z_0}$ is open in $\mbb{D}$.

\medskip

Let $F = \alpha^{-1}(E)$. It suffices to show that if $F$ is non-empty, then either $F$ is a discrete set or $F = \mbb{D}$. Suppose $F$ is not discrete and $z \in D$ is a limit point of $F$. Since $F$ is closed, $z \in F$. Consider the point $\alpha(z) = p \in E$. Since $E$ is an analytic subset of $M$, there exist local defining functions $f_1, f_2, \ldots, f_m$ for $E$ near $p$, i.e., there exists a neighourhood $V \subset M$ of $p$ such that $V \cap E = \{f_1 = f_2 = \ldots = f_m = 0\}$. Let $r >0$ be such that $\alpha(D(z, r)) \subset V$. Consider the holomorphic functions $f_1 \circ \alpha\vert_{D(z,r)}, f_2 \circ \alpha\vert_{D(z,r)}, \ldots, f_m \circ \alpha\vert_{D(z,r)}$ that are defined on the disc $D(z,r)$. Since $z$ is an accumulation point of the zero set $\mathcal{Z}(f_i \circ \alpha\vert_{D(z,r)})$ in $D(z,r)$, it follows that $f_i \circ \alpha\vert_{D(z,r)} \equiv 0$ on $D(z,r)$ for all $1 \le i \le m$. Thus, $D(z,r) \subset F$ and a connectedness argument shows that $F = \mbb{D}$.

\medskip
If $L = \emptyset$, then $\alpha(\mathbb D) \subset E$. Suppose that $q \in \alpha(\mbb{D}) \subset E$ is such that $\mathcal{F}$ is transversal type at $q$. We need to show that $\alpha$ is a constant map. Suppose not. Let $z_0 \in \mbb{D}$, $r > 0$ be such that $\alpha(z_0) = q$ and $\alpha(D(z_0,r)) \subset U_{q}$. Since $\alpha\vert_{D(z_0, r)} : D(z_0, r) \rightarrow U_q \cap E$ is non-constant, let $z_1 \in D(z_0,r)$ be such that $\alpha'(z_1) \neq 0$, and since $\alpha_{n}'(z_1) \rightarrow \alpha'(z_1)$, therefore $\alpha'(z_1) \in C_{\alpha(z_1)}\mathcal{F}$. Also since $\alpha(\mbb{D}) \subset E$, therefore $\alpha'(z_1) \in C_{\alpha(z_1)}E$, which is a contradiction to the transversality property of $\cal{F}$ in the neighourhood $U_{q}$. Therefore, $\alpha \equiv q \in E$. 
\end{proof}

\medskip

By using Lemma \ref{L:LimitUnif} in conjunction with exactly the same set of ideas as in the proof of Proposition $3$ in \cite{NM} gives the following set of equivalent conditions for the continuity of $\eta$ on $M \setminus E$. The details are omitted.

\begin{thm}\label{T:EquCont}
Let $\mathcal{F}$ be a hyperbolic SHFC on a complex manifold $M$, with singular set $E \subset M$. Also, let $g$ be a given hermitian metric on $M$, and $\eta$ be the modulus of uniformization map of $\mathcal{F}$. Suppose that $\mathcal{U}$ is NCP. Then the following are equivalent:

\begin{enumerate}
\item $\eta$ is continuous in $M \setminus E$.
\item For any sequence $\{\alpha_n\}_{n \ge 1}$ in $\mathcal{U}$, which converges in the compact parts of $\mbb{D}$ to some $\alpha : \mbb{D} \rightarrow M$, and $p = \alpha(0) \notin E$, then $\alpha(\mbb{D}) \subset L_{p}$, where $L_p$ is the leaf of $\mathcal{F}$ passing through $p$. 
\item For any sequence $\{\alpha_n\}_{n \ge 1}$ in $\mathcal{U}$, which converges in the compact parts of $\mbb{D}$ to some $\alpha : \mbb{D} \rightarrow M$, and $p = \alpha(0) \notin E$, then $\alpha$ is a uniformization of $L_p$.
\item For any sequence $\{\alpha_n\}_{n \ge 1}$ in $\mathcal{U}$, which converges in the compact parts of $\mbb{D}$ to some $\alpha : \mbb{D} \rightarrow M$, and $p = \alpha(0) \in E$, then $\alpha(\mbb{D}) \subset E$.
\end{enumerate}
\end{thm}

\medskip
\no The next result shows that the transversality condition is sufficient for the continuous extension of the map $\eta$ to the singular part.

\begin{thm}\label{T:ContExt}
Let $\mathcal{F}$ be a hyperbolic {\sf shfc} on a complex manifold $M$ with singular set $E \subset M$. Suppose that $\mathcal{U}$ is NCP and $\eta$ is continuous on $M \setminus E$. If $\cal{F}$ is transversal type at some $p \in E$, then $\eta$ extends continuously to the point $p \in E$. In fact, if $\cal{F}$ is transversal type at each point of $E$, then $\eta$ has a continuous extension to all of $M$.
\end{thm}

\begin{proof}
Define $\tilde{\eta} : M \rightarrow [0, +\infty)$ by $\tilde{\eta}(x) = \eta(x)$ for $x \in M \setminus E$, and $\tilde{\eta}(x) = 0$ for $x \in E$. Suppose that $\tilde{\eta}$ is not continuous at $p \in E$. Then there exist a sequence of points $\{p_n\}_{n \ge 1}$ in $M \setminus E$ such that $p_n \rightarrow p$, and $\eta(p_n) > \epsilon$ for some $\epsilon > 0$ and all $n \ge 1$. Let $\alpha_n$ be a uniformization of the leaf $L_{p_n}$ with $\alpha_n(0) = p_n$. Since $\mathcal{U}$ is NCP, there exists a subsequence $\{\alpha_{n_k}\}_{k \ge 1}$ which converges uniformly on compact subsets of $\mbb{D}$ to a map $\alpha : \mbb{D} \rightarrow M$. Since $\alpha(0) = p \in E$, Theorem~\ref{T:EquCont} shows that $\alpha(\mbb{D}) \subset E$. If $\alpha$ is non-constant, then it would contradict the transversality condition of $\mathcal{F}$ at $p \in E$ as in Lemma \ref{L:LimitUnif}. Therefore $\alpha \equiv p \in E$, and
\[
0 = \vert \alpha'(0)\vert^2 = \lim_{k \to +\infty}{\vert \alpha_{n_k}'(0)\vert^2} = \lim_{k \to +\infty}{\eta(p_{n_k})}
\]
which is a contradiction as $\eta(p_{n_k}) > \epsilon$ for all $n \ge 1$. Therefore, $\eta$ extends continuously to $p \in E$. If $\cal{F}$ is transversal type everywhere, the above argument shows that $\tilde{\eta}$ is continuous on $M$.

\end{proof}

\medskip
\noindent { \it Remark:} If the singular set $E$ of a hyperbolic {\sf shfc} $\mathcal{F}$ is discrete, then $C_{p}E = \{0\}$ for all $p \in E$. Therefore, $\overline{C_{p} \mathcal{F}} \cap C_{p}E = \{0\}$ for all $p \in E$ and this means that such an $\mathcal F$ is always of transversal type. In this case, $\eta$ always extends continuously to $E$ if it is known to be continuous on $M \setminus E$.


\no Consider the locus of discontinuity of $\eta$
\[
D_{\cal{F}} = \{p \in M \sm E : \eta \; \text{is discontinuous at} \; p\}
\]
corresponding to a hyperbolic {\sf shfc} $\mathcal F$.

\medskip

\begin{thm}\label{T:DiscSet}
Let $(M, \cal{F}, E)$ be as above and suppose that $\cal{U}$ is NCP. Then 
\begin{enumerate}
\item For each leaf $L$ of $\cal{F}$, the set $D_{\cal{F}} \cap L$ is either empty or  open in $L$.
\item If $\cal{F}$ is transversal type at $p \in E$ and $\eta$ does not extend continuously to $p$, then $D_{\cal{F}} \neq \emptyset$. In fact, there exists a sequence $p_n \in D_{\cal{F}}$ such that $p_n \to p$.
\end{enumerate}
\end{thm}

\begin{proof} (1) Let $p \in D_{\cal{F}}$. Let $\{p_n\}_{n \ge 1} \subset M \sm E$ be such that $p_n \to p$ but $\eta(p_n) \not\to \eta(p)$. Up to a subsequence, suppose that $\eta(p_n) \to m \neq \eta(p)$.

\medskip
\no For each $p_n$, let $\alpha_n : \mbb{D} \to L_{p_n}$ be a corresponding uniformizer, i.e., $\alpha_n(0) = p_n$, and $\eta(p_n) = \vert \alpha_n'(0)\vert^2$. Since $\cal{U}$ is NCP, and $\alpha_n(0) = p_n \to p$, up to a further subsequence we can suppose that $\alpha_n \to \alpha : \mbb{D} \to M$, with $\alpha(0) = p$, and $\vert \alpha'(0)\vert^2 = m \neq 0$ (using lower semi-continuity of $\eta$, $\eta(p_n) \to m \ge \eta(p) \neq 0$).

\medskip
\no Using the local structure of leaves, there exist $r >0$ such that $\alpha : \Delta_{r} \to L_{p}$, where $\Delta_r = \{z \in \mbb{D} : \vert z\vert < r\}$.

\medskip
Let $\beta : \mbb{D} \to L_{p}$ be a uniformization map such that $\beta(0) = p$, and $\eta(p) = \vert \beta'(0)\vert^2$. Consider $f_1, f_2 : \Delta_r \rightarrow \mbb{R}$ defined by
\[
f_1(z) = (1- \vert z\vert^2)^2 \lvert \alpha'(z)\rvert^2 \; \text{and} \; \; f_2(z) = (1- \vert z\vert^2)^2 \lvert \beta'(z)\rvert^2
\]
and note that $f_2(z) = \eta(\beta(z))$ for $z \in \Delta_r$. Since $f_1, f_2$ are continuous and $f_1(0) \neq f_2(0)$ (as $f_1(0) = \vert \alpha'(0)\vert^2 = m \neq \eta(p) = \vert \beta'(0)\vert^2 = f_2(0)$), there exists $0 < R < r$, such that 
\begin{equation}
f_1(\Delta_R) \cap f_2(\Delta_R) = \emptyset.
\end{equation}

\medskip
\no {\it Claim:} There exist $0 < r_1, r_2 < R$ such that 
\begin{equation}
\beta(\Delta_{r_2}) \subset \alpha(\Delta_{r_1}).
\end{equation}

\no Let $0 < \hat{r} < R$ be such that both $\alpha\vert_{\Delta_{\hat{r}}}, \beta\vert_{\Delta_{\hat{r}}}$ are biholomorphisms into their image. This is possible since $\vert \alpha'(0)\vert, \vert \beta'(0)\vert \neq 0$. Note that $U = \alpha(\Delta_{\hat{r}}) \cap \beta(\Delta_{\hat{r}})$ is open in $L_p$ and contains $p$. Now choose $r_1 < \hat{r}$ such that $\alpha(\Delta_{r_1}) \subset U$, and corresponding to $r_1$, it is possible to choose $r_2 < \hat{r}$ such that $\beta(\Delta_{r_2}) \subset \alpha(\Delta_{r_1}) \subset U$. Thus, the claim is verified.

\medskip

Consider $V = \beta(\Delta_{r_2})$, which is an open set in $L$ containing $p$. Let $q \in V$ and suppose that $q = \beta(z)$ for some $z \in \Delta_{r_2}$. Using $(1.2)$, there exists a $\tilde{z} \in \Delta_{r_1}$ such that $q = \beta(z) = \alpha(\tilde{z})$. Now consider the sequence $q_n := \alpha_n(\tilde{z}) \in M \sm E$. Since $\alpha_n \to \alpha$ uniformly on compact subsets of $\mbb{D}$, $q_n \to q$ since $q_n = \alpha_n(\tilde{z}) \to \alpha(\tilde{z}) = \beta(z) = q$. Also, $\vert \alpha_n'(\tilde{z})\vert \to \vert \alpha'(\tilde{z})\vert$, which in-turn gives
\begin{equation}
(1- \vert \tilde{z}\vert^2)^2 \lvert \alpha_n'(\tilde{z})\rvert^2 \to (1- \vert \tilde{z}\vert^2)^2 \lvert \alpha'(\tilde{z})\rvert^2.
\end{equation}

\no Since $\eta(q_n) = (1- \vert \tilde{z}\vert^2)^2 \lvert \alpha_n'(\tilde{z})\rvert^2$, $(1.3)$ above shows that 
\[
\eta(q_n) \to f_1(\tilde{z})
\]
where $\tilde{z} \in \Delta_{r_1} \subset \Delta_{R}$. Now using $(1.1)$, and the fact that $f_2(z) = \eta(q)$, we get
\[
f_1(\tilde{z}) \neq f_2(z) = \eta(q).
\]
Therefore, $\eta(q_n) \not\to \eta(q)$ although $q_n \to q$. Thus, $q \in D_{\cal{F}} \cap L$, and since $q \in V$ was arbitrary, $V \subset D_{\cal{F}} \cap L$. 

\medskip

\no (2) Since $\eta$ does not extend continuously to $p \in E$, there exist a sequence $q_n \in M \sm E$ such that $q_n \to p \in E$, and $\lim_{n \to +\infty} \eta(q_n) = k \neq 0$. Let $\alpha_n : \mbb{D} \to L_{q_n}$ be the
uniformizer of the leaf $L_{q_n}$ such that $\alpha_n(0) = q_n$. Up to a subsequence, we can suppose that $\alpha_n \to \alpha$ where $\alpha : \mbb{D} \to L \cup E$, for some leaf $L$ of $\cal{F}$, such that $\alpha(0) = p$ and $\vert\alpha'(0)\vert^2 = k$. The transversality condition at $p \in E$ and the fact that $\vert\alpha'(0)\vert^2 = k \neq 0$ together imply the existence of an $r>0$ such that $\alpha(D(0,r) \setminus \{0\}) \subset L \subset M \sm E$. 

\medskip

Now along the leaf $L$, the map $\eta(q)$ is smooth and tends to $0$ as $q \to p$. Therefore, for $\ep > 0$, we can choose $r > 0$ small enough so that 
\[
\vert \eta(q)\vert^2 < \ep.
\]
for all $q \in \alpha(D(0,r) \setminus \{0\})$. For $q = \alpha(z_0) \in \alpha(D(0,r)^{*})$, take $\tilde{q}_n = \alpha_n(z_0)$. Clearly, $\tilde{q}_n \to q$. But 
\[\eta(\tilde{q}_n) = \eta(\alpha_n(z_0)) = \vert \alpha_n'(z_0)\vert^2 (1 - \vert z_0\vert^2)^2 \to \vert \alpha'(z_0)\vert^2 (1 - \vert z_0\vert^2)^2.
\]
We can take $r >0$ (with $\epsilon < k$) small enough so that $\vert \alpha'(z_0)\vert^2 (1 - \vert z_0\vert^2)^2 > 2\epsilon$. Therefore,
\[
\eta(\tilde{q}_n) \to \vert \alpha'(z_0)\vert^2 (1 - \vert z_0\vert^2)^2 > 2\epsilon > \eta(q)
\]
and this shows that $q \in D_{\cal{F}}$, and since $q \in \alpha(D(0,r)^{*})$ was arbitrary, it follows that $\alpha(D(0,r) \setminus \{0\}) \subset D_{\cal{F}} \subset M \sm E$. Now we can take $p_n \in \alpha(D(0,r) \setminus \{0\})$ such that $p_n \to p$ and this concludes the proof.

\end{proof}

\medskip
\no {\it Remark:} Theorem~\ref{T:DiscSet} also shows that there are no isolated points in $D_{\cal{F}}$.

\medskip

\begin{prop}\label{P:ContSepa}
Let $(M, \cal{F}, E)$ be as before and suppose $\mathcal{U}$ is NCP. If a leaf $L \subset M$ of $\mathcal{F}$ is not a local separatrix at any singular point $q \in E$, then $\eta$ is continuous at each $p \in L$.
\end{prop}

\begin{proof}
Let $p \in L \cap D_{\cal{F}}$. Then there exists a sequence $\{p_n\}_{n \ge 1} \subset M \sm E$ such that $p_n \to p$ but $\eta(p_n) \not\to \eta(p)$. Up to a subsequence, suppose $\eta(p_n) \to m \neq \eta(p)$. Let $\alpha_n$ be a uniformizer of the leaf $L_{p_n}$ with $\alpha_{n}(0) = p_n$. Since $\cal{U}$ is NCP and $\alpha_n(0) \to p$, by passing to a further subsequence, there exists $\alpha : \mbb{D} \rightarrow M$ such that $\alpha_n \to \alpha$ uniformly on compact subsets of $\mbb{D}$. By Lemma~\ref{L:LimitUnif}, there is a leaf $\tilde L$ such that $\alpha(\mbb{D}) \subset \tilde{L} \cup E$, with $\alpha^{-1}(E)$ being either discrete (possibly empty) or all of $\mbb{D}$. Now since $\alpha(0) = p \in M \sm E$, therefore $\tilde{L} = L_p = L$, and $\alpha^{-1}(E)$ is discrete. If $\alpha^{-1}(E) \neq \emptyset$, i.e., there exists a $z_0 \in \alpha^{-1}(E)$, then $L$ will be a local separatrix near $q = \alpha(z_0) \in E$, which is a contradiction. Thus $\alpha^{-1}(E) = \emptyset$, and $\alpha(\mbb{D}) \subset L$, which in turn gives
\[
\eta(p_n) = \vert \alpha_n'(0)\vert^2 \to \vert \alpha'(0)\vert^2 \le \eta(p),
\]
that is $m \le \eta(p)$. But by the lower semi-continuity of $\eta$, $m \ge \eta(p)$. Therefore, $m = \eta(p)$, which is a contradiction. Hence $L \cap D_{\cal{F}} = \emptyset$.

\end{proof}

\begin{thm}\label{T:SingHyper}
Let $(M, \cal{F}, E)$ be a hyperbolic {\sf shfc} such that $\mathcal{U}$ is NCP.
Suppose that for each $q \in E$, there exists a neighbourhood $U_q$, and a local $\cal{F}-$invariant hypersurface $\Sigma_q = \{ f_q = 0\}$ such that 
\[
q \in \Sigma_{q} \subset U_{q}.
\]
Then $\eta$ is continuous on $M \sm \bigcup_{p \in (\cup_{q \in E} \Sigma_q)} L_p$.
\end{thm}

\begin{proof}
Let $\tilde{p} \in (M \sm \bigcup_{p 
\in (\cup_{q \in E} \Sigma_q)} L_p) \cap D_{\cal{F}}$. Let $\{p_n\}_{n \ge 1} \subset M \sm E$ be a sequence such that $p_n \to \tilde{p}$, and $\eta(p_n) \not\to \eta(\tilde{p})$. Up to a subsequence, suppose $\eta(p_n) \to m \neq \eta(\tilde{p})$. Let $\alpha_n$ be a uniformizer of the leaf $L_{p_n}$ with $\alpha_{n}(0) = p_n$. Since $\cal{U}$ is NCP and $\alpha_n(0) \to \tilde{p}$, up to a further subsequence, there exists $\alpha : \mbb{D} \rightarrow M$ such that $\alpha_n \to \alpha$ uniformly on compact subsets of $\mbb{D}$. By Lemma~\ref{L:LimitUnif}, there exists a leaf $\tilde L$ such that $\alpha(\mbb{D}) \subset \tilde{L} \cup E$, with $\alpha^{-1}(E)$ being either discrete (possibly empty) or whole $\mbb{D}$. Since $\tilde{p} = \alpha(0) \in M \sm \bigcup_{p 
\in (\cup_{q \in E} \Sigma_q)} L_p \subset M \sm E$, therefore $\tilde{L} = L_{\tilde{p}}$ and $\alpha^{-1}(E)$ is discrete.

\medskip

Now if $\alpha^{-1}(E) \neq \emptyset$, fix $z \in \alpha^{-1}(E)$ (where $\alpha(z) = q \in E$) and let $r>0$ be such that $D(z,r) \sm \{z\} \subset \mbb{D} \sm \alpha^{-1}(E)$ and $\alpha(D(z,r)) \subset U_{q}$. By the choice of $\tilde{p}$ and the fact that $\alpha(D(z,r)) \cap U_{q} \cap E = \{q\}$, we can choose $r> 0$ small enough so that $\alpha(D(z,r)) \cap \Sigma_{q} = \{q\}$. As each $\alpha_n(D(z,r))$ is contained in the leaf $L_{p_n}$, by $\cal{F}-$invariance of $\Sigma_{q}$ and the fact that $\alpha_n\vert_{D(z,r)}$ converges uniformly to $\alpha\vert_{D(z,r)}$, we get that $f_{q} \neq 0$ on $\alpha_{n}(D(z,r))$. But $f_{q} \circ \alpha_{n}(z) \to 0$. Therefore, by Hurwitz's theorem, $f_{q} \circ \alpha\vert_{D(z,r)} \equiv 0$, which is a contradiction. Thus, $\alpha^{-1}(E) = \emptyset$, and $\alpha(\mbb{D}) \subset L_{\tilde{p}}$. This gives $m = \eta(\tilde{p})$, which is again a contradiction to the assumption that $m \neq \eta(\tilde{p})$. Thus, $\eta$ is continuous on $M \sm \bigcup_{p 
\in (\cup_{q \in E} \Sigma_q)} L_p$.

\end{proof}

\no Using Theorem~\ref{T:SingHyper} repeatedly, the following two corollaries are obtained. The details are omitted.

\medskip

\begin{cor}\label{C:MultiHyper}
Let $(M, \cal{F}, E)$ be a hyperbolic {\sf shfc} such that $\mathcal{U}$ is NCP.
Suppose that for each $q \in E$, there exists a neighbourhood $U_q$, and $k_{q}$ local $\cal{F}-$invariant hypersurfaces $\Sigma_{q,1}, \Sigma_{q,2}, \ldots, \Sigma_{q,k_q}$, with $k_q \ge 1$, such that 
\[
q \in \cap_{i=1}^{k_q} \Sigma_{q,i} \subset U_{q}.
\]
Then $\eta$ is continuous on $M \sm \bigcup_{p 
\in (\cup_{q \in E} (\cap_{i = 1}^{k_q}\Sigma_{q,i}))} L_p$.

\end{cor}

\medskip

\begin{cor}\label{C:ContHyper}
Let $(M, \cal{F}, E)$ be a hyperbolic {\sf shfc} such that $\mathcal{U}$ is NCP.
Suppose that for each $q \in E$, there exists a neighbourhood $U_q$, and $m-k$ local $\cal{F}-$invariant hypersurfaces $\Sigma_{q,1}, \Sigma_{q,2}, \ldots, \Sigma_{q,m-k}$ such that 
\[
q \in \cap_{i=1}^{m-k} \Sigma_{q,i} \subset U_{q} \cap E.
\]
Then $\eta$ is continuous on $M \sm E$.

\end{cor}

\medskip
\no {\it Remark}: Corollary~\ref{C:ContHyper} strengthens a result of Fornaess--Sibony (Theorem $20$ in \cite{FS}), who proved the result for singular Riemann surface laminations on a compact Hermitian manifold $M$ with discrete singular set and an additional hypothesis - \textit{no image of $\mbb{C}$ is locally contained in leaves outside the singular set}. It is not too difficult to see that Theorem~\ref{T:SingHyper} still holds if we assume $M$ to be compact Hermitian manifold and replace the hypothesis of $\cal{U}$ being NCP by the one used in \cite{FS} namely, no image of $\mbb{C}$ is locally contained in leaves outside the singular set.

\medskip

\begin{exam}\label{E:E1.14}
Consider the hyperbolic {\sf shfc} $\cal{F}_{X}$ on the unit polydisc $M = \mbb{D}^3 \subset \mbb{C}^3$ induced by the vector field
\[
X = x \frac{\partial}{\partial x} + e^{z}y \frac{\partial}{\partial y}.
\]
The singular set of $\cal{F}_{X}$ is given by $E = \{(0,0,z) \in M\}$. Let $p = (0,0,z) \in E$ be an arbitrary point in $E$. By Example \ref{E:E1.4}, $\cal{F}_{X}$ is of transversal type.

\medskip

\no Now observe that both $\Sigma_1 = \{x=0\}, \Sigma_2 = \{y = 0\}$ are $\cal{F}_{X}-$invariant. Since $\cal{U}$ is NCP and $E = \Sigma_1 \cap \Sigma_2$, Corollary~\ref{C:ContHyper} shows that the modulus of uniformization map $\eta$ is continuous on $M \sm E$. Now Theorem~\ref{T:ContExt} tells us that the map $\eta$ extends continuously to all of $M$.
\end{exam}

\medskip

\begin{exam}\label{E:E1.15}
Consider the vector field on $M = \mbb{D}^3 \subset \mathbb C^3$ given by
\[
X = z \frac{\partial}{\partial x} + xy \frac{\partial}{\partial y} + xy \frac{\partial}{\partial z}.
\]
Denote by $\mathcal{F}_{X}$ the corresponding {\sf shfc} induced by $X$ on $M$ with singular set $E = \{(0,y,0) : y \in \mbb{D}\} \cup \{(x,0,0) : x \in \mbb{D}\}$. It is clear that $\mathcal{F}_{X}$ is a hyperbolic foliation. Also note that $\mathcal{F}_X$ is an example of a holomorphic foliation which is not of transversal type since $(1,0,0) \in \overline{T_{0}\mathcal{F}_{X}} \cap T_{0}E$.

\medskip

One can check that the plane $H = \{y = 0\}$ is $\mathcal{F}_{X}-$invariant, since on the plane $H$, the vector field is given by $X = z \frac{\partial}{\partial x}$. The leaf passing through a point $p = (0,0,z) \in H$, with $z \neq 0$, is given by $L_{p} = \{(\xi,0,z): \xi \in \mbb{D}\}$. Consider the map $\alpha_{p} : \mbb{D} \rightarrow L_{p}$ given by $\alpha_{p}(\xi) = (\xi, 0, z)$. It is easy to see that $\alpha_{p}$ is a uniformization of the leaf $L_{p}$ with $\alpha_{p}(0) = p$, and therefore 
\[
\eta(p) = \vert \alpha'(0)\vert^2 = 1.
\] 
Now consider the sequence $p_n = (0,0,\frac{1}{n}) \in H \setminus E$, and let $\alpha_{n}$ be the corresponding uniformizer of the leaf $L_{p_n}$ with base point $p_n$ given by $\alpha_n(\xi) = (\xi, 0, \frac{1}{n})$. Clearly, $\alpha_n \to \alpha$ uniformly on compact subsets of $\mbb{D}$, where $\alpha : \mbb{D} \rightarrow M$ is given by $\alpha(\xi) = (\xi, 0, 0) \in E$, i.e., $\alpha(\mbb{D}) \subset E$. Now $p_n \to 0 \in E$, and $\eta(p_n) = 1$ for all $n \ge 1$. 

\medskip

Consider the map $g : \mbb{D} \to M$ given by $g(w) = (w, \frac{w^2}{2},\frac{w^2}{2})$. Observe that $g$ is injective and $g(\mbb{D}) \cap E = \{0\}$. Also $g'(w) = (1, w, w) \in T_{g(w)}\mathcal{F}_{X}$, for all $w \in \mbb{D} \setminus \{0\} = \mbb{D}^{\ast}$. Therefore $g(\mbb{D}^{\ast})$ should be contained in a leaf of $\mathcal{F}_{X}$ and since $g(w) \to \partial M$ as $w \to \partial \mbb{D}$, we get that $g(\mbb{D}^{\ast})$ is a leaf $L$ of $\mathcal{F}_{X}$ which is biholomorphic to $\mbb{D}^{\ast}$. Thus we can take any sequence $q_n \in L$ such that $q_n \to 0$ and $\eta(q_n) \to 0$ (since the Poincar\'e metric on $\mbb{D}^{\ast}$ is complete at the origin). 

\medskip

Hence, there are two sequences $\{p_n\}_{n \ge 1}, \{q_n\}_{n \ge 1}$ in $M \setminus E$ such that $p_n,q_n \to 0$, but $\eta(p_n) \to 1$, and $\eta(q_n) \to 0$. Thus, $\eta$ does not extend continuously to $0 \in E$.
\end{exam}

\begin{exam}\label{E:E1.16}
Let $\cal{F}_{X}$ be the {\sf shfc} induced by the vector field 
\[
X = x\frac{\partial}{\partial x} + zy \frac{\partial}{\partial y} 
\]
on the open polydisc of radius $r = (r_1,r_2,r_3)$ denoted by $M= P(0,r)$. The singular set of $\cal{F}_{X}$ is $E = \{(0,y,z) \in M : yz = 0\} = \{y-\text{axis}\}\cup \{z-\text{axis}\}$ (here dim$(E) = 1$). Since $M$ is Kobayashi hyperbolic, $\cal{F}_{X}$ is a hyperbolic foliation, and $\cal{U}$ is NCP. 

\medskip

\no It can be checked that $\Sigma_1 = \{x=0\}, \Sigma_2 = \{y=0\}$ and $\Sigma_3 = \{z=0\}$ are all $\cal{F}_{X}-$invariant. For $p = (0,y,0) \in E$, take $U_{p} = M$, and observe that $p \in \Sigma_{1} \cap \Sigma_{3} = \{y- \text{axis}\} \subset E$. Similarly, for $q = (0,0,z) \in E$, we take $U_{q} = M$ and observe that $q \in \Sigma_{1} \cap \Sigma_{2} = \{z- \text{axis}\} \subset E$. Corollary~\ref{C:ContHyper} now gives us the continuity of $\eta$ on $M \sm E$. 

\medskip
\no Observe that $X(0,y,z) = zy \frac{\partial}{\partial y}$, which implies $(0,1,0) \in T_{(0,y,z)}\cal{F}_{X}$ for all $y,z \neq 0$. Take $p = (0,y,0) \in E$ and since $C_{p}E = \{y-\text{axis}\}$, it follows that $(0,1,0) \in C_{p}E$. Take $p_n = (0,y,\frac{1}{n}) \in M \sm E$ and observe that $p_n \to p$ and $(0,1,0) \in T_{p_n}\cal{F}_{X}$ for all $n$. Therefore, $(0,1,0) \in C_{p}\cal{F}_{X} \cap C_{p}E$. Thus, $\cal{F}_{X}$ is not of transversal type at each $p \in \{y-\text{axis}\} \subset E$. 

\medskip

Let $p = (0,y,0) \in E$ be such a point. Again consider $p_n = (0,y, \frac{1}{n}) \in M \sm E$. On $\Sigma_{1} = \{x=0\}$, the vector field reduced to $X = zy\frac{\partial}{\partial y}$, and we can check that the leaf passing through $p_n$ is given by $L_{p_n} = \{(0,\xi,\frac{1}{n}) : \xi \in D(0,r_2)^{\ast}\} \cong \mbb{D}^{\ast}$. Therefore $\eta(p_n) = \eta(p_m) \neq 0$ for all $n, m \in \mbb{N}$. 

\medskip

Now the foliation $\cal{F}_{X}$ on the hypersurface $\Sigma_{3}$ is induced by the vector field $X = x\frac{\partial}{\partial x}$. It can be seen that the leaf passing through a point $(x,y,0) \in M \sm E$ is given by $L = \{(\xi,y,0) : \xi \in D(0,r_1)^{\ast}\} \cong \mbb{D}^{\ast}$. Note that $L \cup \{p\} \cong \mbb{D}$, that is $L$ is a separatrix. Therefore, if we take $q_n \in L$ be such that $q_n \to p$, then $\eta(q_n) \to 0$.

\medskip

Thus, we get two sequences $\{p_n\}_{n \ge 1}, \{q_n\}_{n \ge 1} \subset M \sm E$ such that $p_n \to p$, $q_n \to p$, but $\eta(p_n) \to k \neq 0$ and $\eta(q_n) \to 0$. Hence, $\eta$ does not have a continuous extension to any $p \in E_1 = \{(0,y,0) : y \in D(0,r_2)\} \subset E$.

\medskip
\no Now let $p = (0,0,z) \in E$, where $z \neq 0$. Looking at Example \ref{E:E1.4}, we see that $\cal{F}_{X}$ is transversal type at $p$, and in turn $\cal{F}_{X}$ is transversal type at each $p \in E_2 = \{(0,0,z) : z \in D(0,r_3)^{\ast}\} \subset E$. By Theorem~\ref{T:ContExt}, the map $\eta$ extends continuously to the set $M \cup E_2$.

\end{exam}

\medskip

\begin{exam}\label{E:E1.17}
Let $\cal{F}_{X}$ be the hyperbolic {\sf shfc} induced by the vector field 
\[
X = x\frac{\partial}{\partial x} + zy \frac{\partial}{\partial y} + zy \frac{\partial}{\partial z}
\]
on the open polydisc $M = P(0, r) \subset \mathbb C^3$ of radius $r = (r_1,r_2,r_3)$. The singular set of $\cal{F}_{X}$ is $E = \{(0,y,z) \in M : yz = 0\} = \{y-\text{axis}\}\cup \{z-\text{axis}\}$ which is one-dimensional. Again, one can check that $\cal{U}$ is NCP and that $\Sigma_1 = \{x=0\}, \Sigma_2 = \{y=0\}$ and $\Sigma_3 = \{z=0\}$ are all $\cal{F}_{X}-$invariant. As before, we can use Corollary~\ref{C:ContHyper} here to conclude that $\eta$ is continuous on $M \sm E$. Example \ref{E:E1.5} shows  that $\cal{F}_{X}$ is transversal type. Therefore, $\eta$ can be extended continuously to all of $M$.

\end{exam}

\medskip

\begin{exam}\label{E:E1.18}
Let $\cal{F}_{X}$ be the hyperbolic {\sf shfc} induced by the vector field 
\[
X = xy\frac{\partial}{\partial x} + zy \frac{\partial}{\partial y} + zx \frac{\partial}{\partial z}
\]
on the open polydisc of radius $r = (r_1,r_2,r_3)$ as in the previous example. The singular set of $\cal{F}_{X}$ is given by $E = \{(x,y,z) \in M : xy = yz = zx = 0\} = \{x-\text{axis}\} \cup \{y-\text{axis}\} \cup \{z-\text{axis}\}$ which is one-dimensional. It can be checked that $\cal{U}$ is NCP, and $\Sigma_1 = \{x=0\}, \Sigma_2 = \{y=0\}$ and $\Sigma_3 = \{z=0\}$ are all $\cal{F}_{X}-$invariant. Let
\[
E_1 = \{x-\text{axis}\} = \Sigma_{2} \cap \Sigma_{3}, \; E_2 = \{y-\text{axis}\} = \Sigma_{1} \cap \Sigma_{3} \; \& \; E_3 = \{z-\text{axis}\} = \Sigma_1 \cap \Sigma_2. 
\]
By Corollary~\ref{C:ContHyper}, $\eta$ is continuous on $M \sm E$. 

\medskip

\no Take $p= (x,0,0) \in E$ for $x \neq 0$. Observe that $C_{p}E = \langle e_1\rangle$, and therefore $(1,0,0) = e_1 \in C_{p}E$. Consider the sequence $p_n = (x,\frac{1}{n},0) \in M \sm E$ and note that $e_1 \in T_{p_n}\cal{F}_{X}$ for all $n$. Therefore, $e_1 \in \overline{C_{p}\cal{F}_{X}}$ and hence $\cal{F}_{X}$ is not transversal type at $p$. Since $p$ is arbitrary, $\cal{F}_{X}$ is not transversal type at any point of $\{(x,0,0) : x \in D(0,r_1)^{\ast}\} \subset E$.

\medskip

\no Now on $\Sigma_3 = \{z = 0\}$, the foliation $\cal{F}_{X}$ is given by $X = xy \frac{\partial}{\partial x}$. One can check that the leaf of $\cal{F}_{X}$ passing through $p_n$ is given by $L_{p_n} = \{(\xi,\frac{1}{n},0): \xi \in D(0,r_1)^{\ast}\} \cong \mbb{D}^{\ast}$ for all $n$. Therefore $\eta(p_n) = \eta(p_m) \neq 0$ for all $n,m \in \mbb{N}$. 

\medskip

\no On $\Sigma_{2} = \{y =0\}$, the foliation $\cal{F}_{X}$ is given by $X= xz\frac{\partial}{\partial z}$. Here the leaf of $\cal{F}_{X}$ passing through $\tilde{p} = (x,0,z)$ is given by $L = \{(x,0,\xi) : \xi \in D(0,r_3)^{\ast}\} \cong \mbb{D}^{\ast}$. Note that $L \cup \{p\} \cong \mbb{D}$, i.e., $L$ is a separatrix through $p$. For any sequence $\{q_n\}_{n \ge 1} \subset L$ such that $q_n \to p$, we get $\eta(q_n) \to 0$.

\medskip

\no Therefore, we get two sequences $p_n, q_n \to p$, but $\eta(p_n) \to k \neq 0$, and $\eta(q_n) \to 0$. Thus, $\eta$ does not extend continuously to $p$.

\medskip
We can use similar arguments to check that for every $p \in \tilde E = E \sm\{0\}$, $\cal{F}_{X}$ is not transversal type at $p$ (in fact $\cal{F}_{X}$ is not transversal type for any $p \in E$). Once again, by looking at the structure of leaves in the invariant hyperplanes $\Sigma_{i}$'s, we can conclude that $\eta$ does not extend continuously to any point $p \in \tilde{E}$.

\end{exam}


\section{$\eta$ as a domain functional}

For a domain $U \subset M$, let $\mathcal F_U$ denote the restriction of the foliation $\mathcal F$ to $U$. For a $p \in M \setminus E$, let $L_{p, U}$ be the connected component of $L_P \cap U$ containing $p$. If it is non-empty, $L_{p, U}$ is the leaf of $\mathcal F_U$ containing $p$. Let $\eta_U : U \to (0, \infty)$ be the modulus of uniformization associated to $\mathcal F_U$. Note that $\eta_U$ is defined using the family $\mathcal O(\mathbb D, \mathcal F_U)$. It is of interest to study the variation $U \mapsto \eta_U$. The results in \cite{NM}, which primarily dealt with $\eta_U$ as $U$ increases monotonically, were strengthened in \cite{GV} in which the domains $U$'s were allowed to vary in the Hausdorff sense.

\medskip

To clarify this, let $d$ be the distance on $M$ induced by $g$, and for $S \subset M$ and $\ep > 0$, let $S_{\ep}$ be the $\ep$-thickening of $S$, with distances being measured using $d$. Recall that the Hausdorff distance $\mathcal H(A, B)$ between compact sets $A, B \subset M$ is the infimum of all $\ep >0$ such that $A \subset B_{\ep}$ and $B \subset A_{\ep}$. For bounded domains $U, V \subset M$, the prescription $\rho(U, V) = \mathcal H(\ov U, \ov V) + \mathcal H(\pa U, \pa V)$ defines a metric (see \cite{Bo}) on the collection of all  bounded open subsets of $X$ with the property that if $\rho(U, U_n) \ra 0$, then every compact subset of $U$ is eventually contained in $U_n$, and every neighbourhood of $\ov U$ contains all the $U_n$'s eventually.

\medskip

Under suitable conditions, Theorem $1.1$ in \cite{GV} shows that if $\{U_n\}_{n \ge 1}$ is a sequence of bounded domains in $M$ that converge to $U$ in the sense that $\rho(U, U_n) \to 0$, then $\eta_{U_n} \to \eta_U$ uniformly on compact subsets of $U \setminus E$. The question that remains is what happens near $E$? 

\begin{thm}\label{T:DomainFunctional}
Let $(M, E, \mathcal F)$ be a hyperbolic {\sf shfc}. Let $\{U_n\}, U$ be bounded domains in $M$. Assume that $\rho(U, U_n) \ra 0$, $U$ is taut and there is a taut domain $V$ containing $\ov U$, and $\eta_U$ is continuous. If $\mathcal F_U$ is of transversal type, then $\eta_{U_n} \to \eta_U$ uniformly on compact subsets of $U$ (define $\eta_{U_n} = \eta_{U} \equiv 0$ on the set $E$). 
\end{thm}

The main steps in the proof are exactly the same as that in Theorem $1.1$ in \cite{GV}. The difference lies in using Lemma \ref{L:LimitUnif} (more precisely, the ideas in its proof) and the transversality condition to understand how the limit of a sequence of uniformizations of leaves of $\mathcal F_{U_n}$ behaves near $E$. 

\begin{proof}

To present the central ideas in brief, let $p_n \in U \setminus E$ be a sequence that converges to $p \in U \setminus E$. Let $\al_n : \mbb D \ra L_{p_n, U_n} \subset U_n$ be a uniformization map with $\al_{n}(0) = p_n$. Since $\rho(U, U_n) \ra 0$, $V$ eventually contains all the $U_n$'s and hence the family $\{\al_n\}_{n \ge 1}$ is normal. By passing to a subsequence, let $\ti \al : \mbb D \ra \ov V$ be a holomorphic limit with $\ti \al(0) = p$. Combining Lemma \ref{L:LimitUnif} with the proof of Theorem $1.1$ in \cite{GV} shows that $\ti \al : \mbb D \to L_{p, U}$ is a uniformization map.

\medskip

To show that the convergence is uniform on compact subsets of $U$, it suffices to prove that for $p \in E$, and $p_n \in U \sm E$ such that $p_n \to p$, 
\[
\eta_{U_n}(p_n) \to \eta_{U}(p) = 0.
\]
If not, then $\vert\eta_{U_n}(p_n)\vert > \ep$. Consider the uniformization map $\alpha_n$ of the leaf $L_{p_n, U_n}$ with $\alpha_n(0) = p_n$. Therefore, after passing to a subsequence we see that 
$\{\alpha_n\}_{n \ge 1}$ converges locally uniformly to $\tilde{\alpha} : \mbb{D} \rightarrow U$ with $\tilde{\alpha}(0) = p \in E$. If $\tilde{\alpha} \equiv p$, then $\vert \eta_{U_n}(p_n)\vert < \ep$ for large $n$, which is a contradiction. Therefore, $\tilde{\alpha}$ is a non-constant map. Since $\mathcal{F}_{U}$ is of transversal type, there exist $z_0 \in \mbb{D}$ such that $q = \tilde{\alpha}(z_0) \in U \setminus E$. Now take $\phi \in \text{Aut}(\mbb{D})$ such that $\phi(0) = z_0$. Observe that $\alpha_n \circ \phi$ is a uniformization of the leaf $L_{p_n, U_n}$ with $\alpha_n \circ \phi(0) = q_n \in U \setminus E$, and it converges locally uniformly to $\hat{\alpha} : \mbb{D} \rightarrow U$, with $\hat{\alpha}(0) = q \in U \setminus E$. By the above reasoning, $\hat{\alpha}$ should be a uniformization of the leaf $L_{q,U}$. But $\hat{\alpha} \equiv \tilde{\alpha} \circ \phi$, which is a contradiction, since $\tilde{\alpha} \circ \phi(\mbb{D})$ intersects $E$. Thus, $\eta_{U_n}$ converges to $\eta_{U}$ uniformly on compact subsets of $U$.

\end{proof}


\section{Local regularity of $\eta$} 

Let $\mathcal{F}$ be a hyperbolic {\sf shfc} on a complex manifold $M$ with singular set $E \subset M$. Following \cite{NM}, for $p \in E$, we will say that:
\begin{enumerate}
\item $\cal{F}$ satisfies $(P.1)$, if there exist a neighbourhood $U \subset M$ of $p$, such that the modulus of uniformization map $\eta_{U}$ of $\cal{F}\vert_{U}$ is continuous on $U \setminus E$.
\item $\cal{F}$ satisfies $(P.2)$, if there exist a neighourhood $\ti U \subset M$ of $p$, such that the metric $\frac{4g}{\eta_{\ti U}}$ is complete at each $q \in \ti U \cap E$ in $\ti U \sm E$.
\end{enumerate}

When $E$ is discrete, it is shown in \cite{NM} that $(P.2)$ implies $(P.1)$. This actually holds in general also.

\begin{thm}\label{T:LocalRegularity}
Let $\mathcal{F}$ be a hyperbolic {\sf shfc} on a domain $W \subset \mbb{C}^n$, with singular set $E \subset \mbb{C}^n$. Take $p \in E$. If $\cal{F}$ satisfies $(P.2)$ around $p$, then $\cal{F}$ satisfies $(P.1)$ around $p$.
\end{thm}

\begin{proof}
Let $U \subset W \subset \mbb{C}^n$ be a domain such that $p \in U$ and $\frac{4g}{\eta_{U}}$ is complete on $E$ in $U \sm E$. Let $B = B_r = \{\vert z-p\vert < r\} \subset U$, and $\Lambda_{B} = \frac{4g}{\eta_{B}}$. Since $\eta_B \le \eta_U$, therefore $\Lambda_B \ge \Lambda_U$, which in turn shows that $\Lambda_B$ is complete on $E$ in $B \sm E$.

\medskip

Since $B$ is taut, $\cal{U}_B$ is NCP. Therefore, according to Theorem~\ref{T:EquCont}, it suffices to prove that if $\{\alpha_{n}\}_{n \ge 1} \subset \cal{U}_B$ converges locally uniformly on $\mbb{D}$ to some $\alpha : \mbb{D} \rightarrow B$, where $q = \alpha(0) \in E \cap B$, then $\alpha(\mbb{D}) \subset E \cap B$.
\medskip

Let $\{\alpha_{n}\}_{n \ge 1} \subset \cal{U}_B$ be such that $\alpha_n(0) \to q \in E \cap B$. Fix $p_0 \in B \setminus E$, and let 
\[
\overline{D_r} = \{p \in B \sm E : d_P(p,p_0) \le r\}
\]
where $d_P$ denotes the distance induced by the metric $\Lambda_B$ on $B$.

Since $\Lambda_B$ is complete on $E$ in $B \sm E$, we have $\cup_{r >0} \overline{D_r} = B \sm E$. Let $E_q$ be the connected component of $E \cap B$ containing $q$, and $W_r$ be the component of $B \sm \ov{D_r}$ containing $E_q$.

\no Clearly $W_{r_1} \subset W_{r_2}$ if $r_1 \ge r_2$. Also since $\cap_{r>0} B \sm \ov{D_r} = E \cap B$, therefore $\cap_{r>0} W_r = E_q$.

\medskip

\no {\it Claim}: Given $0 <\rho < 1$, and $r >0$, there exists $n_0 \ge 1$ such that if $n \ge n_0$ then $\alpha_n(\ov{\Delta_{\rho}}) \subset W_r$, where $\ov{\Delta_{\rho}} = \{z \in \mbb{D} : \vert z\vert \le \rho\}$.

\medskip
Let $\{\alpha_{n_j}\}_{j \ge 1}$ be a convergent subsequence which converges locally uniformly on $\mbb{D}$ to $\alpha : \mbb{D} \rightarrow B$, such that $\alpha(0) = q \in E \cap B$. Assuming the above claim, we get $\alpha(\ov{\Delta_{\rho}}) \subset E_q$, for all $\rho \in (0,1)$. Thus $\alpha(\mbb{D}) \subset E_q \subset E \cap B$.

\medskip
\no {\it Proof of the Claim:} Let $d_L$ denotes the distance induced by $d_p$ on a leaf $L$, and $d_{\Delta}$ denotes the Poincar\'e distance on unit disc $\mbb{D}$. For $\beta \in \cal{U}_B$ and $z_1, z_2 \in \mbb{D}$
\[
\vert \beta(z_2), \beta(z_1)\vert \le d_P(\beta(z_2), \beta(z_1)) \le d_L(\beta(z_2), \beta(z_1)) \le d_{\Delta}(z_2, z_1).
\]

Fix $\rho \in (0,1)$, and $r >0$. Let $c = d_{\Delta}(0, \rho)$ and $r_1 = c+r$. Since $\alpha_n(0) \to q$, there exist $n_0 \ge 1$ such that $\alpha_n(0) \in W_{r_1}$ for $n \ge n_0$. Therefore for $z \in \ov{\Delta_{\rho}}$, and $n \ge n_0$,
\[
d_P(\alpha_n(z), p_0) \ge d_P(\alpha_n(0), p_0) - d(\alpha_n(0), \alpha_n(z)) > r_1 - d_{\Delta}(0, z) \ge r_1 - c = r.
\]

Since $\alpha_n(\ov{\Delta_{\rho}})$ is connected and $\alpha_n(0) \in W_r$ for all $n \ge n_0$, therefore $\alpha_n(\ov{\Delta_{\rho}}) \subset W_r$, for all $n \ge n_0$.

\end{proof}

\medskip
\no The next example is motivated by Proposition 5 in \cite{NM}.

\begin{exam}\label{E:E3.2}
Consider the holomorphic vector field in $\mbb{C}^3$ defined by
\[
X(z) = (X^1 + \text{higher terms}, X^2 + \text{higher terms}, 0)
\]
where $X^1, X^2$ are homogeneous polynomial of degree $k$ in the variables $z_1, z_2$ (thinking of the variable $z_3$ as a parameter). By restricting the domain of $X$, assume that $0 \in \mbb{C}^2$ is an isolated singularity of $(X^1, X^2)$, for each fixed $z_3$. Let $U = B_{\rho} \times \Delta_{\rho}$ be the domain on which $X$ satisfies these conditions; here $B_{\rho} = \{ z \in \mbb{C}^2 : \vert z\vert < \rho\}$, and $\Delta_{\rho} = \{z \in \mbb{C}: \vert z\vert < \rho\}$. Consider the hermitian metric $h$ on $U$ given by
\[
h(z)= \frac{\vert \pi(z)\vert^{2k-2}}{\vert X(z)\vert^2 \log^2{\frac{\vert \pi(z)\vert}{r}}} \vert dz\vert^2,
\]
where $r > \rho$, and $\pi(z_1, z_2, z_3) = (z_1, z_2)$. Let $\cal{F}_{U}$ be the holomorphic foliation induced by $X$ on $U$, with the singular set being $E = \{z \in U : z_1 = z_2 = 0\}$. Since the leaves of the foliation $\cal{F}_U$ are contained in the level sets $\{z_3 = c\}$, it can be seen that $\cal{F}_U$ is of transversal type.

\medskip

Since $(X^1, X^2)$ has isolated singularity at $(0,0)$ for each $z_3 \in \Delta_{\rho}$, therefore we can choose $\rho >0$ small enough such that there exist a constant $C > 0$ which satisfies
\[
C^{-1} \lvert \pi(z)\rvert^k \le \lvert X(z)\rvert \le C \lvert \pi(z)\rvert^k
\]

\no Calculations in \cite{NM} shows that there exist constant $K > 0$ such that 
\[
\Lambda_{U} \ge \frac{K}{\vert \pi(z)\vert^2 \log^2{\frac{\vert \pi(z)\vert}{r}}} \lvert dz\rvert^2,
\]
where $\Lambda_{U}$ denotes the leafwise Poincar\'e metric of $\cal{F}_U$. Since the metric $\frac{K}{\vert \pi(z)\vert^2 \log^2{\frac{\vert \pi(z)\vert}{r}}} \lvert dz\rvert^2$ is complete at $E = \{\pi(z) = 0\}$ in $U$, therefore $\Lambda_U$ is also complete at $E$ in $U$. Hence $\cal{F}_U$ satisfies $(P.2)$ around $0 \in U$. Thus, the above result tells us that there exists a neighourhood $V$ of $0 \in \mbb{C}^3$ such that $\Lambda_V$ is continuous on $V$.

\end{exam}


\no\textbf{Acknowledgement:} The authors would like to thank Viêt-Anh Nguyên for posing the question about the behaviour of the map $\eta$ near the singular set $E$. 


\end{document}